\documentclass[11pt,a4paper]{amsart}
\usepackage{amssymb,amsmath,amsthm}
\usepackage{commath}
\usepackage{mathtools}
\usepackage{mathrsfs}
\usepackage{accents}

\usepackage[utf8]{inputenc}

\usepackage{caption}
\usepackage{subcaption}

\usepackage[colorlinks]{hyperref}
\hypersetup{
    allcolors = black
}

\newtheorem{theorem}{Theorem}

\newtheorem{lemma}[theorem]{Lemma}

\newtheorem{conjecture}[theorem]{Conjecture}

\newtheorem{prop}[theorem]{Proposition}
\newtheorem{claim}[theorem]{Claim}

\newcommand{\R}{\mathbb{R}}
\newcommand{\eps}{\varepsilon}
\newcommand{\dd}{\,\mathrm{d}}

\newcommand{\Sc}{\mathcal{S}}
\newcommand{\Lc}{\mathcal{L}}
\newcommand{\x}{\mathbf{x}}

\DeclareMathOperator\arctanh{arctanh}

\makeatletter
\@namedef{subjclassname@2020}{%
  \textup{2020} Mathematics Subject Classification}
\makeatother

\author{Gergely Ambrus, Imre Bárány, Péter Frankl, Dániel Varga}

\thanks{
Research of GA was partially supported by ERC Advanced Grant "GeoScape no. 882971,  by the Hungarian National Research grant no. NKFIH KKP-133819,  and by project no. TKP2021-NVA-09. Project no.
TKP2021-NVA-09 has been implemented with the support provided by the
Ministry of Innovation and Technology of Hungary from the National
Research, Development and Innovation Fund, financed under the
TKP2021-NVA funding scheme.
Research of IB was partially supported by Hungarian National Research grants no. 131529, 131696, and 133819. Research of DV was supported by the Hungarian Ministry of Innovation and Technology NRDI Office within the framework of the Artificial Intelligence National Laboratory Program,  by the European Union project RRF-2.3.1-21-2022-00004 within the framework of the Hungarian Artificial Intelligence National Laboratory, and Hungarian National Excellence Grant 2018-1.2.1-NKP-00008.}

\title{Piercing the chessboard}

\keywords{Cells in a lattice, lines, discrete plank problems}

\subjclass[2020]{Primary 11H31, secondary 05B40, 52C30}

\date{\today}
\begin{document}

\maketitle

\begin{abstract}
We consider the minimum number of lines $h_n$ and $p_n$ needed to intersect or pierce, respectively, all the cells of the $n \times n$ chessboard. Determining these values can also be interpreted as a strengthening of the classical plank problem for integer points.  Using the symmetric plank theorem of K. Ball, we prove that $h_n = \lceil \frac n 2 \rceil$ for each $n \geq 1$. Studying the piercing problem, we show that $0.7n \leq  p_n \leq n-1$ for $n\geq 3$, where the upper bound is conjectured to be sharp. The lower bound is proven by using the linear programming method, whose limitations are also demonstrated.
\end{abstract}

\maketitle

\section{Cells and lines}
How many lines are needed to pierce each cell of the $n \times n$ chessboard? Likewise, what is the minimum number of lines required to intersect every cell? These innocent-looking questions serve as targets of the present note.

To start with, we introduce some notations. For $ n \geq 1$, let $Q_n$ denote the $n \times n$ chessboard embedded in $[-1,1]^2$. Its {\em cells} are the closed squares
\begin{equation}\label{cijdef}
c_{ij} = \left[ -1 + (i-1)\cdot \frac 2 n , -1 + i \cdot \frac 2 n\right] \times  \left[-1 + (j-1) \cdot \frac 2 n , -1 + j  \cdot \frac 2 n \right]
\end{equation}
with $i, j \in [n]$, where $[n] = \{ 1, \ldots, n \}$. A line $\ell \subset \R^2$ is said to {\em hit} or {\em intersect} a cell $c_{ij}$ if $\ell \cap  c_{ij} \neq \emptyset$, and it {\em pierces} $c_{ij}$ if $\ell \cap  \textrm{int} \, c_{ij} \neq \emptyset$. Let $h_n$ and $p_n$ be the minimal number of lines needed to hit or pierce, respectively,  each cell of $Q_n$.

The question of determining $p_n$ was raised by Bárány and Frankl in~\cite{BF21, BF21+}.
It turns out that the question has a close connection with the classical plank problem in the plane \cite{B51, T32}. In particular, the celebrated symmetric plank theorem of K. Ball~\cite{B91} implies that given any set of $n-1$ lines, there always exists a point $(x,y)$ in $[0,n-1]^2$ such that the interior of the cell $[x, x+1] \times [y,y+1]$ is not intersected by any of the lines. The present question boils down to the following (see Conjecture~\ref{conj_piercing}): does there exist an {\em integer point} with the same property? If true, this would provide a significant strengthening of the plank theorem in this special case, and could also initiate the study of  plank problems for lattice points. Even though we are not able to give a complete answer to the above question, we show nontrivial bounds on $p_n$ (see Theorem~\ref{thm_piercing_lower}).


Our first observations are trivial. Clearly, $h_n \leq p_n$ holds for each $n$. Piercing each column of $Q_n$ with a vertical line shows that $p_n \leq n$. More generally, all the cells of $Q_n$ can be pierced by $n$ parallel lines in any given direction which are at distance $\frac 2 n$ from each other in the $\ell_1$ distance and do not go through any grid points.
On a similar note, selecting every second vertical boundary line between the cells of $Q_n$ yields that $h_n \leq \lceil \frac{n}{2} \rceil$.

It is easy to give a sharp upper bound for the number of cells pierced by an arbitrary line. The following simple statement is part of the mathematical folklore (see also \cite{B83}):
\begin{prop}\label{prop_line}
 Every line pierces at most $2n -1$ cells of $Q_n$.
\end{prop}
This readily implies that $p_n \geq \frac{n}{2}$. We note that higher dimensional versions of this estimate have been recently studied by Bárány and Frankl \cite{BF21, BF21+}.

Note that the analogue of Proposition~\ref{prop_line} does not hold for hitting: diagonals of the square $Q_n$ intersect $3n$ cells. Nevertheless, using the  symmetric plank theorem of K. Ball, we prove that the upper bound on $h_n$ given above is sharp:

\begin{theorem}\label{thm_hitting}
  For each $n \geq 1$, $h_n =  \lceil \frac{n}{2} \rceil$ holds.
\end{theorem}

Determining  $p_n$ proves to be more difficult. Surprisingly, and somewhat counter-intuitively, the upper bound $p_n \leq n$ can be improved: there exist several configurations of $n-1$ lines piercing all the cells of~$Q_n$. This is also the subject of a mathematics puzzle \cite{BS19} that appeared in The Guardian.
\begin{theorem}\label{thm_piercing_upper}
$p_2=2$ and for each $n \geq 3$, $p_n \leq n-1$.
\end{theorem}
The lower bound $p_n \geq \frac n 2$ is not sharp either: using linear programming methods, we asymptotically strengthen it. Here comes our main result.
\begin{theorem}\label{thm_piercing_lower}
If $n$ is sufficiently large, then $p_n > 0.7 n $.
\end{theorem}

The gap between the upper and lower estimates for $p_n$ is large, and there is certainly room for improvement. A computer search was carried out in order to find configurations of $n-2$ lines piercing all cells of $Q_n$ when $n \leq 15$, to no avail. Based on this computational evidence, we venture to formulate the following conjecture.

\begin{conjecture}\label{conj_piercing}
  For all $n \geq 3$, $p_n = n-1$.
\end{conjecture}

Note that the Corollary of \cite{B91} implies that given a set of $n-1$ lines, there always exists a {\em translate} of a cell contained within $Q_n$ which is not pierced by any of these lines. Conjecture~\ref{conj_piercing} is the analogue of this statement for lattice cells.

Even though the linear programming method used for proving Theorem~\ref{thm_piercing_upper} may be strengthened, in Section~\ref{section_LPlimit} we demonstrate its limitations. Theorem~\ref{thm_LPupper} states that  this method cannot give a lower bound larger than $0.925 \, n$. The proof of Conjecture~\ref{conj_piercing} will require novel ideas.

We conclude the section with a short list of related works. Keszegh~\cite{K12} studied the  minimal number of segments of a polygonal path covering each vertex of a rectangular grid, and the problem was generalized in \cite{DGKT14} for finite point sets in the plane. Generalizing the problem we study here, Richter~\cite{R22} determined the minimum number of {\em monotonous polyominoes} covering each square of a rectangular chessboard. 

Our research topic is closely related to the field of {\em digital geometry} which aims at defining geometrical objects appearing in digital images composed of pixels, i.e. sets of points in $\mathbb{Z}^2$ (see \cite{R91, KR04}). We are going to build upon its notions and results. Finding the cells of a square grid that intersect a given line segment is of central interest here, thanks to its application in rendering polygons on a rasterized computer screen. The set of cells (pixels) arising from such an intersection is called a \textit{digital line}. Bresenham’s line algorithm \cite{B65} is a traditional method to \textit{rasterize} a line segment, i.e. turn it into a digital line. In the terminology of digital geometry, our question of interest is finding the minimal number of digital lines that cover a given square bitmap. 

\section{Hitting}
In this section we prove Theorem~\ref{thm_hitting}. We use the symmetric plank theorem of K. Ball. A functional defined on a finite-dimensional space $X$ is called {\em unit functional} if its operator norm equals to 1. 

\begin{lemma}[Ball~\cite{B91}]\label{lemma_planksymm}
If $(\varphi_i)_1^m$ is a sequence of unit functionals in a finite-dimensional normed space $X$, $(t_i)_1^m$ is a sequence of reals and $(w_i)_1^m$ is a sequence of positive numbers with $\sum_1^m w_i = 1$ then there is a point $x$ in the unit ball of $X$ for which
\[
|\varphi_i(x) - t_i|\geq w_i
\]
for every $i \in [m]$.
\end{lemma}

\begin{proof}[Proof of Theorem~\ref{thm_hitting}]
By the remark  preceding Proposition~\ref{prop_line}  it suffices to prove $h_n \geq \lceil \frac n 2 \rceil$. Assume on the contrary that there exists a set of $m := \lceil \frac n 2 \rceil -1$ lines,
$\{\ell_1, \ldots, \ell_m \}$, intersecting each cell of $Q_n$. Let  $u_i$ denote the (Euclidean) unit normal of $\ell_i$. Let $X$ be the space $\R^2$ endowed with the $\ell_\infty$-norm, that is, with unit ball $[-1,1]^2$. To each planar vector $u \neq 0$, assign the linear functional
\begin{equation}\label{phidef}
\varphi_u: \ \x \mapsto \frac{  \langle \x, u \rangle }{\|u \|_1}
\end{equation}
which has norm 1 in $X^*$.
Set $\varphi_i := \varphi_{u_i}$ for each $i \in [m]$. Then
\begin{equation}\label{ellidef}
\ell_i = \{\x \in \R^2: \ \varphi_i(\x) = t_i  \}
\end{equation}
holds for every $i \in [m]$ with some $t_i \in \R$.
Note that the set of points $\x \in \R^2$ for which $|\varphi_i(\x) - t_i| \leq w$ holds equals to the union of all (closed) squares of edge length $ 2 w$ whose center lies on~ $\ell_i$.

Set $\eps >0 $ so that $(\frac 2 n + \eps)(\lceil \frac n 2 \rceil - 1) = 1$, and let $w_i = \frac 2 n + \eps$ for every $i$. Lemma~\ref{lemma_planksymm} implies the existence of $\x \in [- 1,1]^2$ for which
\[
|\varphi_i(\x) - t_i|\geq \frac 2 n + \eps
\]
holds for each $i$. This is equivalent to the fact that the open square
\[
\x + \Big(- \frac 2  n  - \eps, \frac 2 n  + \eps \Big)^2
\]
is not met by any of the lines $\ell_i$. We finally observe that any open square of side length strictly greater than $\frac 4 n $ centered in $[-1,1]^2$   contains a cell of~$Q_n$.
\end{proof}

\section{A piercing construction}

To each line $\ell$ we assign the corresponding {\em digital line}  \cite{R91} $\sigma(\ell)$ of cells of $Q_n$ which are pierced by $\ell$ :
\begin{equation}\label{sigmadef}
  \sigma(\ell) = \{ c_{ij}: \ i,j \in [n] \textrm{ and } \ell \textrm{ pierces }  c_{ij} \}
\end{equation}
(see the shaded region on Figure~\ref{fig1}a)). Note that the same concept appears by a variety of names in the literature, e.g. {\em digital straight line segment} \cite{R74}, 
{\em chaincode string} \cite{DS84}, {\em discrete segment} \cite{MI85}; see also the closely related notions of  {\em standard} and {\em naive}  discrete lines \cite{DRR95}.

\medskip
We continue with a nontrivial upper bound on $p_n$.

\begin{proof}[Proof of Theorem~\ref{thm_piercing_upper}]
We will construct a set of $n-1$ lines which pierces every cell of $Q_n$.
Define the line $\ell$ by the equation
\[
y =  (1 - \eps) x
\]
where $\eps >0$ is a small positive number, for instance $\eps=\frac 1{n^2}$ will do.
Thus, $\ell$ is obtained by a small clockwise rotation of the line $y=x$ about the origin, see Figure~\ref{fig1}a).

   Let now
\[
\ell_i = \ell + \left(1 - \frac {2i+1}{n}, -1 + \frac {2i+1}{n} \right)
\]
for each $i \in [n-2]$,  see Figure~\ref{fig1}b). Then $\ell_i$ passes through the center of the cell $c_{(n-i) \, i}$  and intersects its boundary on its two vertical sides (because its slope is slightly less than one). Choose the value of $\eps$  so that $\sigma(\ell_i)$ contains exactly three cells in row $i$, at most $2$ cells in any other row of $Q_n$, and $\ell_i$ does not pass through any vertex of a cell in $Q_n$. Then along with any cell of $\sigma(\ell_i)$, it also contains a horizontal or vertical neighbour of the cell. Since $\ell_{i+1}$ is obtained from $\ell_{i}$ by a translation with $(-\frac 2 n, \frac 2 n )$, this implies that $\bigcup_{i=1}^{n-2}\sigma(\ell_i)$ covers each cell of $Q_n$ in between $\ell_1$ and $\ell_{n-2}$ (simply check the off-diagonal chains of cells). Therefore, the lines $\ell_1, \ldots, \ell_{n-2}$ pierce all the cells of $Q_n$ except $c_{1n},c_{2n}$ and $c_{(n-1)\, 1}c_{n \,1}$ (equivalently, the corresponding digital lines cover all the cells except for these). Clearly these four cells can be pierced by one line $\ell_0$, which leads to a set of $n-1$ lines piercing every cell of $Q_n$.
\end{proof}

\begin{figure}[h]
   \centering
  \includegraphics[width = 0.8 \textwidth]{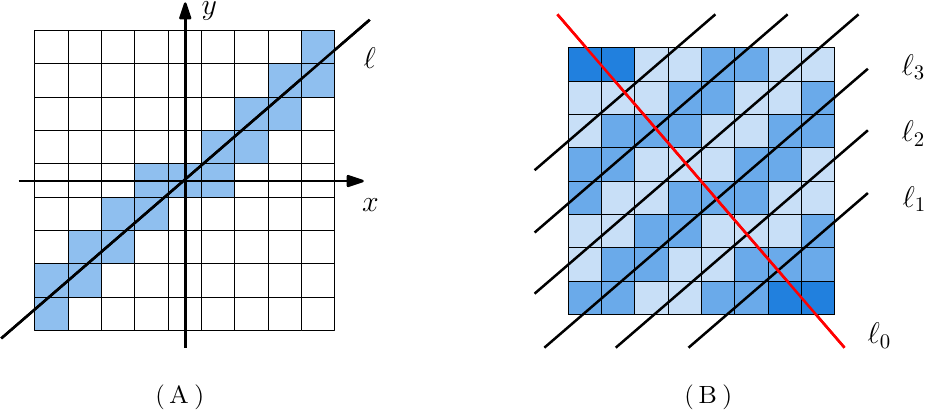}
  \caption{Piercing $Q_n$ with $n-1$ lines. }
    \label{fig1}
\end{figure}

We note that the above construction is not unique. Indeed, the same pattern works using translates of a line which pierces exactly 3 consecutive cells in some row and does not pass through any grid point.
Thus, for any $\alpha \in (\frac 1 3, 1)$ there exists a set of $n-2$ lines of slope $\alpha$, along with a line of slope $-\frac{n-1}{n-2}$ which pierce $Q_n$.

Going even further, we challenge the dedicated reader to find a piercing configuration which consists of $k$ parallel lines of an approximately diagonal direction and $n - 1 - k$ lines in the orthogonal direction, for each $k \in [n-2]$. (The two families of lines have to be positioned in a cross-like pattern, so that every boundary cell is pierced by them.)

\section{Few lines do not pierce}

\begin{proof}[Proof of Proposition~\ref{prop_line}]
 Assume the slope of $\ell$ is non-negative. We move a point on $\ell$ from left to right and order the cells in $\sigma(\ell)$ (cf. \eqref{sigmadef}) in the order the point enters them. If $c_{ij}\in \sigma(\ell)$, then the next cell in this order is either $c_{(i+1)j}$ or $c_{i(j+1)}$. The cells in $\sigma(\ell)$ thus form a zig-zag going right or up at each step.
\end{proof}

The lower bound $p_n \geq \frac n 2$ follows from Proposition~\ref{prop_line}. We are going to improve this lower bound by the linear programming method that yields the following
statement. Below, $\x = (x, y)$, and $\dd \|\x\|_1$ stands for $\dd(|x| + |y|)$.

\begin{lemma}\label{lemma_LPest}
Assume that $\mu: [-1,1]^2 \rightarrow \R_{\geq 0}$ is a Lipschitz continuous density function such that for each line $\ell$ in the plane,
\begin{equation}\label{ellcond}
  \int_{\ell \,\cap\, [-1,1]^2} \mu(\x) \dd \|\x\|_1 \leq 1.
\end{equation}
Then for any $\eps> 0$,
\[
p_n > \Big (  \frac 1 2 \int_{[-1,1]^2} \mu(\x) \dd \x  - \eps \Big ) n
\]
holds if $n$ is sufficiently large.
\end{lemma}

\begin{proof}

Denote by $\Lc$ the set of all lines intersecting $[-1,1]^2$, and set $\Sc_n = \{ \sigma(\ell): \ \ell \in \Lc \}$. So $\Sc_n$ is the set of all  digital lines of $Q_n$.
Clearly, $\Sc_n$ is a finite set, and for every element of $\Sc_n$ there exists a line which pierces the cells therein.
Determining $p_n$ is equivalent to finding the optimal value of the following integer linear program (LP$_i$):

\begin{equation}
\tag*{\hspace{10 pt}(LP$_i$)}
\begin{aligned}
&\textrm{Minimize  }\sum_{\sigma \in \Sc_n} \rho(\sigma) \textrm{ subject to} \\
&\rho(\sigma) \in \{ 0, 1 \} \textrm{ for all } \sigma \in \Sc_n \textrm{, and }\\
&\sum_{\sigma \in \Sc_n:\ c_{ij} \in \sigma} \rho(\sigma) \geq 1 \textrm{  for every } i,j \in [n]
.
\end{aligned}
\end{equation}

Therefore, the optimal value of the following {\em continuous} linear program (LP$_c$) gives a lower bound on~$p_n$:

\begin{equation}
\tag*{\hspace{10 pt}(LP$_c$)}
\begin{aligned}
&\textrm{Minimize  }\sum_{\sigma \in \Sc_n} \rho(\sigma) \textrm{ subject to} \\
&\rho(\sigma) \geq 0 \textrm{ for all } \sigma \in \Sc_n \textrm{, and }\\
&\sum_{\sigma \in \Sc_n:\ c_{ij} \in \sigma} \rho(\sigma) \geq 1 \textrm{  for every } i,j \in [n]
.
\end{aligned}
\end{equation}

Taking  the dual program of (LP$_c$) leads to the following setup.
Let $w: [n]\times[n] \rightarrow [0,1]$  be a weight function, and use the notation $w_{ij} = w(i,j)$. Let $M$ be the solution of the following continuous linear program:

\begin{equation}
\tag*{\hspace{10 pt}(LP$_d$)}
\begin{aligned}
&\textrm{Maximize  }\sum_{i,j=1}^n w_{ij} \textrm{ subject to} \\
&w_{ij} \geq 0 \textrm{ for every } i,j =[n]\textrm{ and }\\
&\sum_{i,j \in [n]: \ c_{ij} \in \sigma} w_{ij}\   \leq 1
\textrm{ for every  } \sigma \in S_n.
\end{aligned}
\end{equation}

\noindent
By weak linear programming duality, the optimal value $M$ of (LP$_d$) gives a lower bound on that of (LP$_c$), which in turn gives a lower bound on the solution of (LP$_i$). Thus,
\begin{equation}\label{pnM}
p_n \geq M.
\end{equation}
(Note that this fact also follows elementarily, without referring to LP duality. The above linear programs depend on $n$ but we suppress this dependence.)

Thus, we face the problem: How to solve (LP$_d$)? Since there are only finitely many  digital lines in $\Sc_n$,  the number of constraints in (LP$_d$) is finite. Therefore, (LP$_d$) can be solved computationally, at least for small values of $n$. In the case $n=30$, the optimal weight distribution on $Q_n$ found by computational methods is plotted on Figure~\ref{fig2}. This yields the estimate $p_n \geq 0.7205n$ for $n=30$.

\begin{figure}[h]
  \centering
  \includegraphics[width = 0.45 \textwidth]{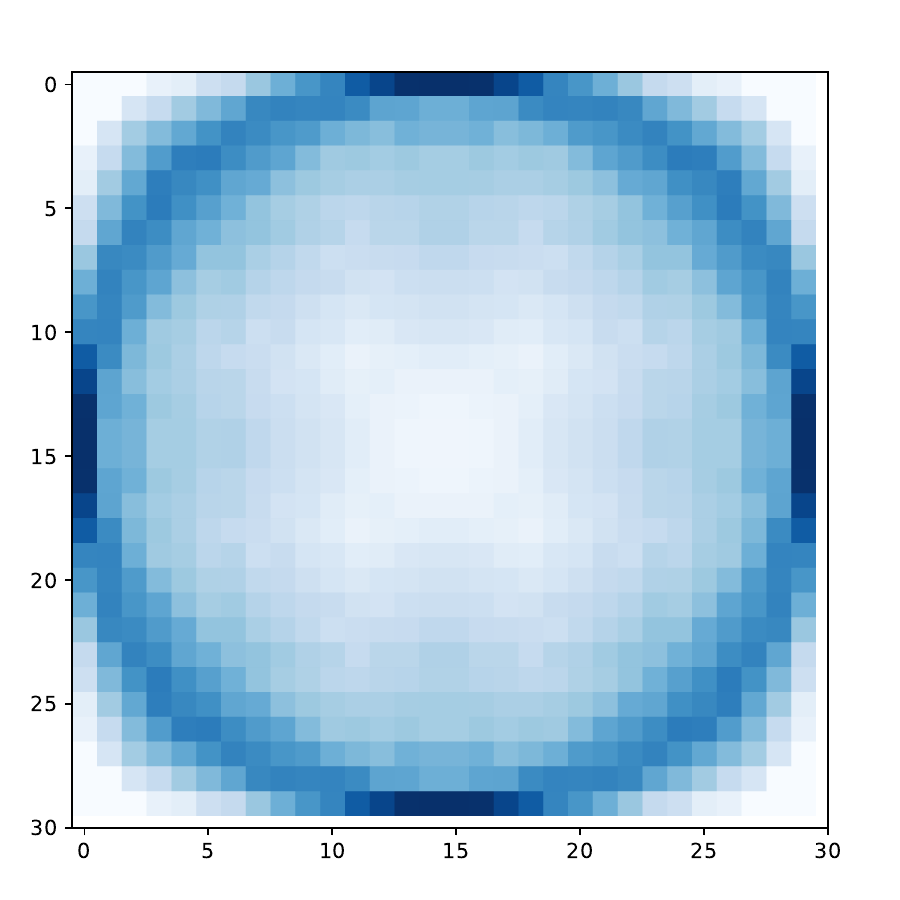}
  \caption{Optimal weight distribution on $Q_{n}$ with $n = 30$, where lighter (white) colors represent values close to 0. }
  \label{fig2}
\end{figure}

However, for large values of $n$ the computational approach breaks down; solving the $n=30$ case already required several hours of computation.
We are going to replace (LP$_d$) by its continuous approximation when $n \rightarrow \infty$.

Let $\mu: [-1,1]^2 \rightarrow \R_{\geq 0}$ be a density function which is Lipschitz continuous with respect to the Euclidean distance, with constant~$\lambda$.
To each cell $c_{ij}$ of $Q_n$ we assign the weight
\begin{equation}\label{wij}
w_{ij} = \frac n 2 \int_{c_{ij}}\mu(\x) \dd \x
\end{equation}
where $\dd \x$ stands for the standard Lebesgue measure.
Note that
\begin{equation}\label{sumwij}
\sum_{i,j=1}^n w_{ij} = \frac n 2 \int_{[-1,1]^2} \mu(\x) \dd \x.
\end{equation}
On the other hand, for each line $\ell$ defined by the equation $\varphi_u(\x) = t$ (cf. \eqref{phidef}), we introduce the corresponding plank
\begin{equation}
P(\ell) = \Big\{ \x \in \Big[-1 - \frac 2 n, 1 + \frac 2 n \Big]^2:\ |\varphi_u(\x) - t| \leq \frac 1 n \Big\}.
\end{equation}
Then $P(\ell)$ is the intersection of $[-1 - \frac 2 n,1+ \frac 2 n]^2$ with the union of the squares of edge length $\frac 2 n$ centered on $\ell$.
Note that $P(\ell)$ is not contained in $Q_n$, a small part of it is outside.

\begin{figure}[h]
  \centering
  \includegraphics[width = 0.5 \textwidth]{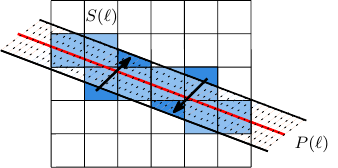}
  \caption{Comparison of the regions $P(\ell)$ and $S(\ell)$}
\label{fig3}
\end{figure}

\begin{claim}
Using the above notations, for each line $\ell \in \Lc$,
\[
\sum_{i,j, \in [n]: \ c_{ij} \in \sigma(\ell)} w_{ij} \leq \frac n 2 \int_{P(\ell)} \mu(\x)  \dd \x +  \frac {34 \lambda} {n} \,.
\]
\end{claim}
\begin{proof}
We may assume that the slope of $\ell$ is non-positive. Then, the upper boundary line of $P(\ell)$ is obtained from the lower boundary line by a translation with $v =(\frac 2 n, \frac 2 n)$.

Let
\[
S(\ell) = \bigcup_{\sigma(\ell)} c_{ij}.
\]
Notice that $S(\ell) \cap( S(\ell) + v )= \emptyset$ and $S(\ell) \cap( S(\ell) - v )= \emptyset$. Observe that if a cell $c_{ij} \in \sigma(\ell)$ reaches below $P(\ell)$, then $( c_{ij} \setminus P(\ell) ) + v \subset P(\ell)$, and similarly, if $c_{ij} \in \sigma(\ell)$ reaches above $P(\ell)$, then $( c_{ij} \setminus P(\ell) ) - v \subset P(\ell)$, unless $c_{ij}$ is at the boundary of $Q'_n$ (see Figure~\ref{fig3}). Thus, the parts of $S(\ell)$ which are not contained in $P(\ell)$ may be moved into $P(\ell)$ by a translation with either $v$ or $-v$, without creating overlaps (note that this is the reason for defining $P(\ell)$ on $[-1 - \frac 2 n,1+ \frac 2 n]^2$  instead of $[-1, 1]^2$). Therefore, by the Lipschitz property of $\mu$,
\begin{align*}
 \sum_{i,j, \in [n]: \ c_{ij} \in \sigma(\ell)} w_{ij} &= \frac n 2 \int_{S(\ell)} \mu(\x) \dd \x
 \leq \frac n 2 \int_{P(\ell)} \left(\mu(\x) +  \frac{2 \sqrt{2}}{n}\lambda \right) \dd \x \\
 &\leq \frac n 2 \int_{P(\ell)} \mu(\x)  \dd \x +\sqrt{2} \lambda \, \textrm{Area}(P(\ell)) \\
 & \leq \frac n 2 \int_{P(\ell)} \mu(\x)  \dd \x + \sqrt{2} \lambda \cdot \frac 2 n \cdot 4 \left(1 + \frac 2 n  \right)\\
 & \leq \frac n 2 \int_{P(\ell)} \mu(\x)  \dd \x  + \frac {34 \lambda}{n}\,.
 \qedhere
\end{align*}
\end{proof}

Now, the proof of the Lemma is easy to complete.
The Lipschitz property of $\mu$ ensures that as $n \rightarrow \infty$,
\[
 \frac n 2 \int_{P(\ell)} \mu(\x)  \dd \x  + \frac {34 \lambda}{n}\rightarrow \int_{\ell \,\cap\, [-1,1]^2} \mu(\x) \dd \|\x\|_1 \,.
\]
Thus, \eqref{ellcond} guarantees that the weights $w_{ij}$ given by \eqref{wij} satisfy the criterion of (LP$_d$), and hence by \eqref{pnM} and \eqref{sumwij}, $\frac n 2 \int_{[-1,1]^2} \mu(\x) \dd \x$ provides a lower bound on $p_n$.
\end{proof}

In view of Lemma~\ref{lemma_LPest}, the next  task is to find a suitable density function on $[-1,1]^2$.

\begin{proof}[Proof of Theorem~\ref{thm_piercing_lower}] In order to illustrate the method, we will first consider the following density function which leads to a slightly weaker estimate:
\begin{equation}\label{mu1}
{\mu_1}(\x) = \frac 3 4 ( x^2-2x^2y^2+y^2)
\end{equation}
(see Figure~\ref{fig4a}).
The function $\mu_1(\x)$ is clearly Lipschitz. We will show that \eqref{ellcond} holds for~$\mu_1$.

Assume that the line $\ell$ is defined by the equation $y = a x + b$.  By the symmetries of $\mu(\x)$, for proving \eqref{ellcond} we may assume that $a \in [0,1]$ and $b \geq 0$. Also, if $\ell$ hits $[-1,1]^2$, then $b \leq 1 + a $ must hold.
Let $\x_1 = (x_1, y_1)$ and $\x_2 = (x_2, y_2)$, $x_1 \leq x_2$ be the points where $\ell$ hits the boundary of $[-1,1]^2$ (these may coincide when $\ell$ hits only a corner). By the assumptions on $a$ and $b$, we have that $x_1 = -1$ and $y_1 \in [-1, 1]$. Depending on  the magnitude of $b$, $\x_2$ may lie on the upper or the right side of $[-1,1]^2$:
\begin{itemize}
\item if $0 \leq b \leq 1 - a$, then $x_2 = 1$ and $y_2 = a + b$;
\item if $1-a \leq b \leq 1 + a$, then $x_2 = \frac {1 -b}{a}$ and $y_2 = 1$.
\end{itemize}
Accordingly, \eqref{ellcond} is equivalent to
\begin{equation}\label{max1}
\max_{a \in [0,1], b \in [0, 1-a]} \, (1+a)\int_{-1}^1 (x^2 - 2 x^2 (a x + b)^2 + (a x + b)^2 ) \dd x \leq \frac 4 3
\end{equation}
and
\begin{equation}\label{max2}
\max_{a \in [0,1], b \in [1-a, 1+a]} \, (1+a)\int_{-1}^{\frac {1 -b}{a}} (x^2 - 2 x^2 (a x + b)^2 + (a x + b)^2 ) \dd x \leq \frac 4 3 \,.
\end{equation}
By evaluating the above integrals, \eqref{max1} reads as
\[
\max_{a \in [0,1], b \in [0, 1-a]} \, (1 + a) \left( \frac 2 3 - \frac{2 a^2}{15} + \frac{2 b^2}{3} \right) \leq \frac 4 3 \,.
\]
It is simple to check that the above function attains its maximum on the given domain at $a= 0, b=1$ with the maximum value being exactly $\frac 4 3$. Similarly, \eqref{max2} amounts to
\[
\max_{a \in [0,1], b \in [1-a, 1+a]} \frac {(1 + a)(-1 + 5 a^2 + 5 a^3 - a^5 + 5 b^2 + 5 a^3 b^2 - 5 b^3 -  5 a^2 b^3 + b^5)}{15 a^3} \leq \frac 4 3 \,.
\]
By analyzing the function above, one obtains that for any given $a \in [0,1]$, the above function is maximized at $b = (-1 - a + \sqrt{9 - 6 a + 9 a^2}) /2$ on the interval $b \in [1-a, 1+a]$. Substituting this value leads to a  function of $a$ which is decreasing on $(0,1]$.

Thus, we obtain that the maximum of $  \int_{\ell \,\cap\, [-1,1]^2} \mu_1(\x) \dd \|\x\|_1 $ is attained for lines $\ell$ which contain a side of $[-1,1]^2$, with the extreme value being 1.  Since
 \[
\int_{[-1,1]^2}  \mu_1 (\x) \dd \x = \frac 4 3,
\]
 Lemma~\ref{lemma_LPest} guarantees that for any $\eps>0$,
\[
p_n > \left( \frac 2 3 - \eps \right) n
\]
if $n$ is sufficiently large.

The stronger estimate $p_n > 0.7 n$ of Theorem~\ref{thm_piercing_lower} can be shown by considering the density function
\begin{equation*}
\mu_2(\x) = 0.3 (|x| + |y|) + 0.43 (|x|^3 + |y|^3) - 0.585 (|x|^3 |y| + |y|^3 |x|) - 0.16 x^2 y^2
\end{equation*}
(see Figure~\ref{fig4b}), which has been found by numerical optimization. Clearly, $\mu_2$ is Lipschitz on $[-1,1]^2$, and
\[
\int_{[-1,1]^2} \mu_2(\x) \dd \x \approx 1.4039.
\]
That \eqref{ellcond} is satisfied for $\mu_2$ is again checked by elementary calculus, although the calculations are more tedious than in the previous case because of the absolute values in the definition of $\mu_2(\x)$. We only note that, using the notation above,  the maximum line integral is taken at $a=0.5612, b=0.5612$ with the maximum value being approximately $0.9971$. Thus, because of the symmetries taken into account, lines of maximal weight go close to a corner of $[-1,1]^2$, and they are of slope around $0.56$, $-0.56$, $1.78$ or $-1.78$. Further calculations may be completed by the aid of a computer algebra software.
\end{proof}

\begin{figure}
\centering
\begin{subfigure}{.5\textwidth}
  \centering
  \includegraphics[width=.7\linewidth]{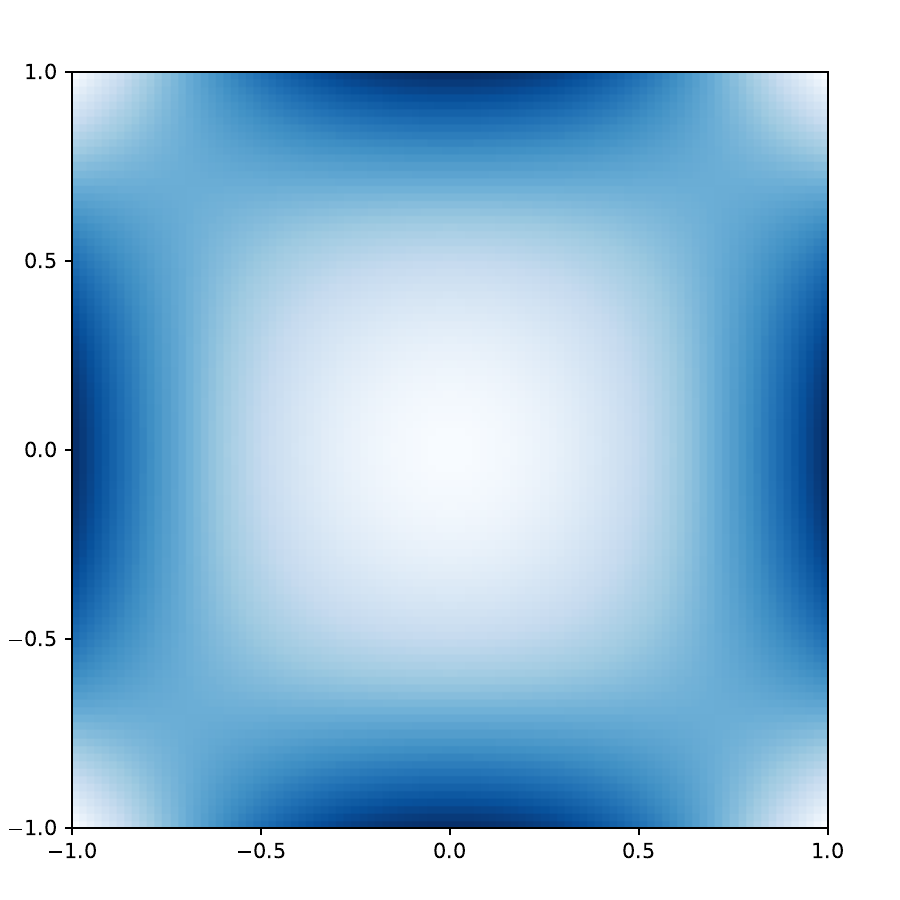}
  \caption{The function $\mu_1(x,y)$}
  \label{fig4a}
\end{subfigure}%
\begin{subfigure}{.5\textwidth}
  \centering
  \includegraphics[width=.7\linewidth]{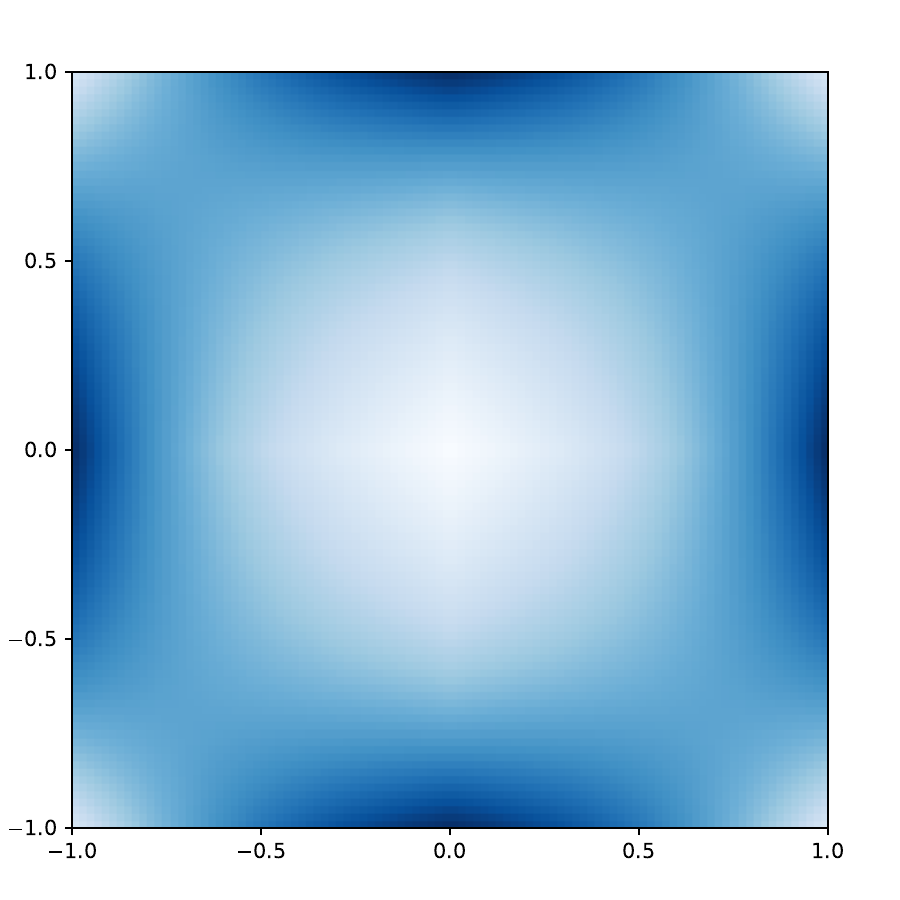}
  \caption{The function $\mu_2(x,y)$}
  \label{fig4b}
\end{subfigure}
\caption{The two density functions used in the proof of Theorem~\ref{thm_piercing_lower}. Darker (blue) colors represent values close to 0.75, while lighter (white) stands for values close to 0.}
\label{fig4}
\end{figure}

\section{Limitations of the linear programming method}
\label{section_LPlimit}

Along the lines of the previous section, one may increase the lower bound on $p_n$ by including higher order terms in the density function, although at the price of increased computational difficulty. However, in this section we prove that the conjectured value $p_n = n-1$ can not be proved by applying the continuous linear programming method. In order to show that, note that any feasible solution of the linear program (LP$_c$) yields an upper bound on the optimal value of (LP$_c$), hence, by linear programming duality, on that of (LP$_d$) as well. Therefore, in order to demonstrate the confinedness of the LP method, it suffices to provide a weight distribution $\rho(\sigma)$ on the set $S$ of  digital lines which satisfies the constraints of (LP$_c$) and for which $\sum_{\sigma \in S} \rho(\sigma) < n-1$. We will, in fact, prove a much stronger bound.
\begin{theorem}
\label{thm_LPupper}
For every sufficiently large $n$, there exists a weight distribution $\rho(\sigma)$ on $\Sc_n$ satisfying the conditions of (LP$_c$) for which
\[
\sum_{\sigma \in \Sc_n} \rho(\sigma) < 0.925 \, n.
\]
\end{theorem}

Recall that $\Lc$ is the set of lines intersecting $[-1,1]^2$. We parameterize $\Lc$ as follows: for each $s \in [-1, 1]$, introduce the vector
\[
u_s = (s, 1- |s|)
\]
and for each $s, t \in [-1,1]$, define
\begin{equation}\label{ellst}
\ell(s,t) = \{ \x \in \R^2: \ \varphi_{u_s}(\x) = t \}
 \end{equation}
(cf. \eqref{phidef}). That is, $\ell(s,t)$ is the line with normal $u_s$  defined by $s x + (1 - |s|) y = t$. It is easy to check that $\ell \in \Lc$  if and only if it may be expressed in the form \eqref{ellst} with $s,t \in [-1,1]$.

Reminiscent of Lemma~\ref{lemma_LPest}, we will construct the sought-after weight distribution on $\Sc$
by applying a continuous approximation.

\begin{lemma}
\label{lemma_LPupper}
Assume that $\nu: [-1,1]^2 \rightarrow \R_{\geq 0}$ is a Lipschitz continuous density function which satisfies that for each $(x_0, y_0) \in [-1,1]^2$,
\begin{equation}
\label{nucond1}
   \int_{-1}^0 \nu (s, y_0 + s(x_0 + y_0) ) \dd s +
\int_{0}^1  \nu (s, y_0 + s(x_0 - y_0) ) \dd s \geq 1.
\end{equation}
Then for any $\eps>0$, the optimal value of {\rm (LP$_c$)} is bounded from above by
\[
\Big (  \frac 1 2 \int_{[-1,1]^2} \nu(s,t) \dd s \dd t  + \eps \Big ) n
\]
if $n$ is sufficiently large.

\end{lemma}

\begin{proof}
To each  digital line $\sigma \in \Sc_n$, assign the weight
\begin{equation}\label{rhodef}
  \rho(\sigma) = \frac n 2 \int_{(s,t) \in [-1,1]^2: \ \sigma(\ell(s,t)) = \sigma  } \nu(s,t) \dd s \dd t.
\end{equation}
Clearly,
\begin{equation}
\label{nuint}
\sum_{\sigma \in \Sc_n} \rho(\sigma) = \frac n 2 \int_{[-1,1]^2} \nu(s,t) \dd s \dd t.
\end{equation}
In order for the weights $\rho(\sigma)$ to satisfy the conditions of (LP$_c$), the following inequality must hold for all $i,j \in [n]$:
\begin{equation}
\label{nucond}
\frac n 2 \int_{(s,t) \in [-1,1]^2: \ c_{ij} \in \sigma(\ell(s,t)) } \nu(s,t) \dd s \dd t
 \geq 1.
\end{equation}
The above integration goes over the set of parameters $(s,t)$ whose corresponding lines pierce~$c_{ij}$. This region can be determined as follows. Let $(x_0, y_0)$ be the center of $c_{ij}$. Recall that the side-length of $c_{ij}$ equals to $\frac 2 n$. First, assume that $s \leq 0$, equivalently, that the slope of $\ell(s,t)$ is non-negative. Then $\ell(s,t)$ pierces $c_{ij}$ if and only if it contains a point of the form $(x_0 + a, y_0 - a)$ with $|a| < \frac 1 n $. That is equivalent to the condition
\[
|s (x_0 + y_0) - t + y_0| < \frac 1 n \,.
\]
Similarly, if $s \geq 0$, then $\ell(s,t)$ intersects $c_{ij}$ iff it goes through a point of the form $(x_0 + b, y_0 + b)$ with $|b| < \frac 1 n $, which is equivalent to
\[
|s (x_0 - y_0) - t + y_0| < \frac 1 n \,.
\]
Thus, we derive that the set of lines in $\Lc$ which pierce $c_{ij}$ is represented on the $(s,t)$-plane by the region $R_{ij}$ which is the vertical parallel neighborhood of radius $\frac 1 n$ of the union of two segments, connecting the points $(-1, -x_0)$ and $(0, y_0)$, and $(0,y_0)$ and $(1, x_0)$, respectively. In particular, all the cross-sections of $R_{ij}$ parallel to the $t$-axis are of length $\frac 2 n$ (see Figure~\ref{fig5}).

\begin{figure}[h]
  \centering
  \includegraphics[width = 0.45 \textwidth]{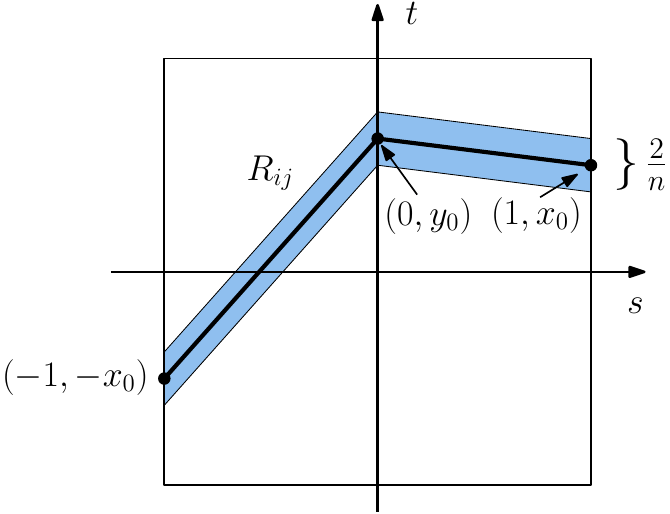}
  \caption{ The region $R_{ij}$ on the $(s,t)$-plane which represents  the set of lines intersecting $c_{ij}$}
\label{fig5}
\end{figure}

We note that using different parametrizations, several algorithms have been given for determining the set of lines which generate a given {\em digital segment} \cite{DS84, MI85, LB93, DA06} (the set in question is referred to as an equivalence class of lines, or the domain of a chaincode string, or the preimage of the digital segment, respectively).
We chose the $(s,t)$-parametrization in order to facilitate the construction and verification of a suitable density function on the space $\Lc$. In particular, the simple geometric structure of the region $R_{ij}$ expedites the forthcoming calculations.


By the Lipschitz property of $\nu$, the integrals in \eqref{nucond} converge uniformly as $n \to \infty$:
\[
\frac n 2 \int_{(s,t) \in [-1,1]^2: \ c_{ij} \in \sigma(\ell(s,t)) } \nu(s,t) \dd s \dd t \rightarrow \int_{-1}^0 \nu (s, y_0 + s(x_0 + y_0) ) \dd s +
\int_{0}^1  \nu (s, y_0 + s(x_0 - y_0) ) \dd s.
\]
Thus, \eqref{nucond1} and  \eqref{nucond} guarantee that the conditions of (LP$_c$) hold for a suitably scaled copy of $\rho$ provided by \eqref{rhodef}, and the statement of Lemma~\ref{lemma_LPupper} follows from  \eqref{nuint}.
\end{proof}

We are left with the task of finding a suitable density function on the $(s,t)$-plane.

\begin{proof}[Proof of Theorem~\ref{thm_LPupper}]
We will construct a density function $\nu(s,t)$ which is not Lipschitz continuous, but satisfies \eqref{nucond1} as well as
\begin{equation}\label{nuint1849}
   \int_{[-1,1]^2} \nu(s,t) \dd s \dd t < 1.849.
\end{equation}
Lemma~\ref{lemma_LPupper} may then be applied to a suitably fine Lipschitz continuous approximation of $\nu$, yielding the estimate of Theorem~\ref{thm_LPupper}.

The density $\nu$ is defined using a parameter $\gamma \in [0,1]$ whose value we will set later.
Let $\eps>0$ be a small positive number. Define $P_1$ to be the parallelogram
\[
P_1 = \mathrm{conv}\{(0,1), (-\eps, 1), (-1, -1), (-1 + \eps, -1 )\}
\]
of area  $2 \eps$.
Let $P_2 = \{(s, t):\ (s, -t) \in P_1\}$ be the mirror image of $P_1$ with respect to the $s$-axis, and set $P_3 = - P_1$ and $P_4 = -P_2$.
Define $P_5$ to be the rectangle
\[
P_5 = \mathrm{conv}\Big \{\Big(- \frac {1 - \eps} {2}, h \Big), \Big(- \frac {1 + \eps} {2}, h \Big),\Big(- \frac {1 + \eps} {2}, -h \Big), \Big(- \frac {1 - \eps} {2}, -h \Big) \Big \}
\]
where $h = \frac {\tan \gamma  }{2}  + \eps$. Note that if $\eps< 1 - \frac{\tan 1}{2}$, then $h\in[0,1)$. Finally, let $P_6 = - P_5$ (see Figure~\ref{fig6a}).

\begin{figure}
\centering
\begin{subfigure}{.5\textwidth}
  \centering
  \vspace{3 pt}
  \includegraphics[height=.675\linewidth]{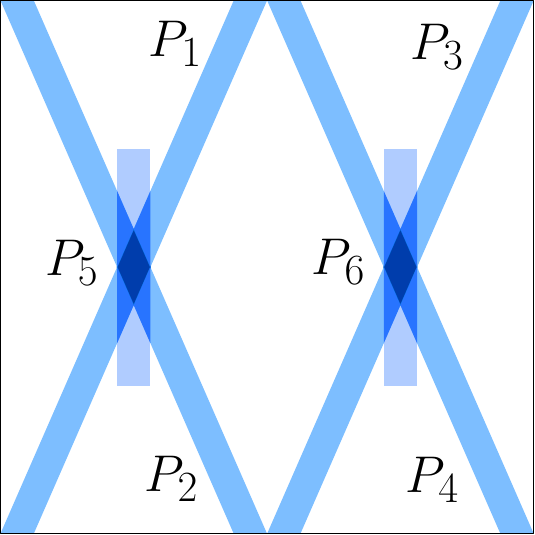}
  \vspace{13 pt}
  \caption{Structure of the function $\nu(s,t)$ \\ \phantom{gh}}
  \label{fig6a}
\end{subfigure}%
\begin{subfigure}{.5\textwidth}
  \centering
  \includegraphics[height=.8\linewidth]{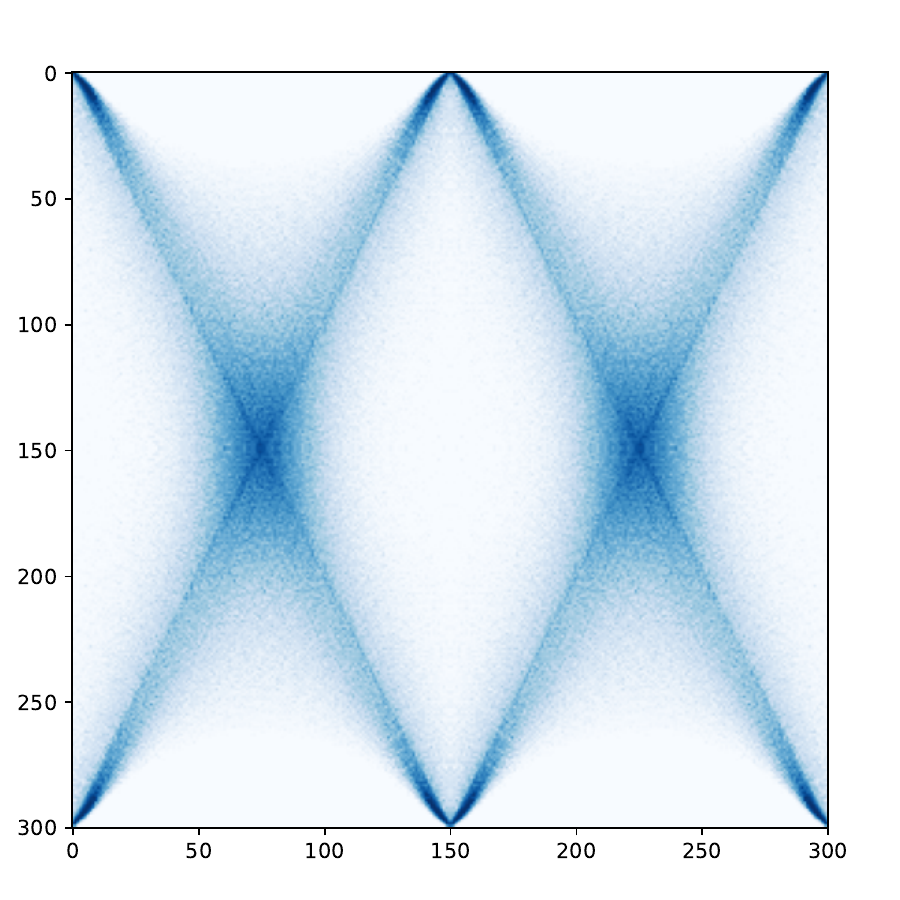}
  \caption{Discrete density function on the $300 \times 300$ grid, found by computer search }
  \label{fig6b}
\end{subfigure}
\caption{Dual density functions on the $(s,t)$-plane}
\label{fig6}
\end{figure}


We will say that a line $\ell$ is of  {\em angle} $\alpha$ if its slope is $\tan \alpha$. Let $\frac \pi 2 - \beta$ be the angle of the line connecting the points $(-1 + \eps/2, -1)$ and $(- \eps/2, 1)$. Then $\tan \beta = \frac {1 - \eps}{2}$.

Now, let $\ell$ be a line of angle $\alpha$ which goes through two points $(-1, t_1)$ and $(0, t_2)$ with $t_1, t_2 \in [-1, 1]$. A simple calculation shows (see Figure~\ref{fig7a}) that the horizontal projection of $\ell \cap P_1$ has length
\[
\frac{\eps \cos \alpha \cos \beta}{\cos(\alpha + \beta)}.
\]
By symmetry, the horizontal projection of $\ell \cap P_2$ has length $\frac{\eps \cos \alpha \cos \beta}{\cos(\alpha - \beta)}$. Let
\begin{equation}\label{fidef}
\phi(\alpha) = \frac{ \cos \alpha \cos \beta}{\cos(\alpha + \beta)} + \frac{ \cos \alpha \cos \beta}{\cos(\alpha - \beta)}.
\end{equation}
It is easy to check that $\phi(\alpha)$ is symmetric, convex on $(- \arctan 2, \arctan 2)$, and it attains its minimum on this interval at 0 with $\phi(0) =2$. In particular, $\phi$ is increasing on $[0, \arctan 2]$.

\begin{figure}
\centering
\begin{subfigure}{.5\textwidth}
  \centering
  \includegraphics[height=.5\linewidth]{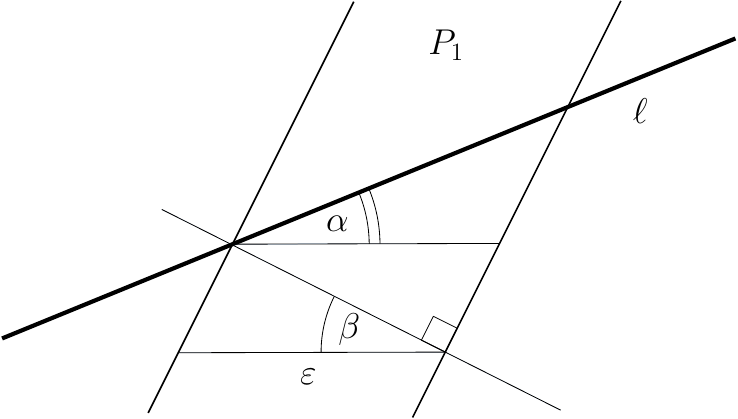}
  \caption{Scheme of $\ell \cap P_1$}
  \label{fig7a}
\end{subfigure}%
\begin{subfigure}{.5\textwidth}
  \centering
  \includegraphics[height=.5\linewidth]{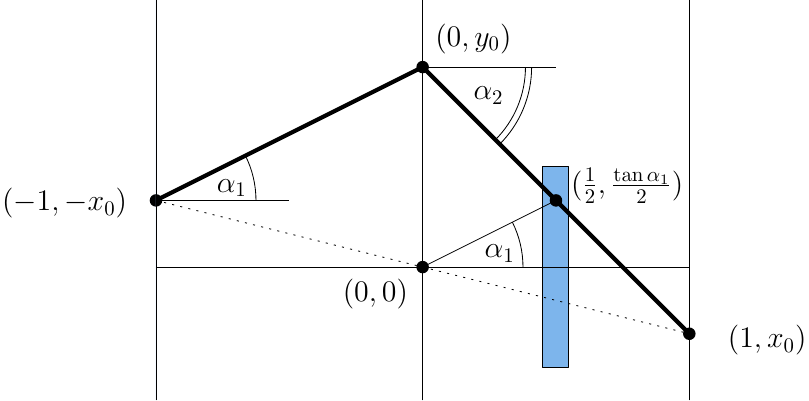}
  \caption{The case $|\alpha_1| + |\alpha_2| < \gamma$}
  \label{fig7b}
\end{subfigure}
\caption{}
\label{fig7}
\end{figure}

Next, we define the density function $\nu (s,t)$. Set $w = \frac 1 {2 \phi (\gamma)}$.
Denote by $\chi_A$ the indicator function of the set $A$, and let
\[
\nu_1 (s,t)= \frac {w} { \eps} \sum_{i=1}^4 \chi_{P_i}(s,t).
\]
Introduce
\[
\psi(t) =  \Big( 1 - \frac{ \phi(\arctan 2 t)}{\phi(\gamma)}\Big).
\]
Then $\psi(t)$ is convex on $(-\frac \pi 2, \frac \pi 2)$ with the maximum value taken at $t = 0$, and $\psi(t) >0$ for $|t| \leq h$. Define
\begin{equation*}
\nu_2 (s,t) =
\begin{cases}
 \frac 1 {\eps} \psi(t) &\textrm{ for }(s,t) \in P_5 \cup P_6 \\
 0 &\textrm{ otherwise,}
\end{cases}
\end{equation*}
and let $\nu(s,t) = \nu_1(s,t) + \nu_2(s,t)$. We will show that $\nu$ satisfies the condition \eqref{nucond1}.

Let  $x_0, y_0 \in [-1, 1]$ be arbitrary, and denote by $\alpha_1$ and $\alpha_2$ the angle of the lines through the points $(-1, -x_0)$ and $(0, y_0)$, and $(0,y_0)$ and $(1, x_0)$, respectively. Then, since $\phi$ is symmetric and convex,
\begin{align}\label{nu1int}
\begin{split}
\int_{-1}^0\nu_1 (s, y_0 &+ s(x_0 + y_0) ) \dd s + \int_{0}^1 \nu_1 (s, y_0 + s(x_0 - y_0) ) \dd s = w ( \phi(\alpha_1) + \phi(\alpha_2)) \\
&\geq 2 w \phi \left( \frac{|\alpha_1| + |\alpha_2|}{2} \right)= \frac{\phi \left( \frac{|\alpha_1| + |\alpha_2|}{2} \right)}{\phi(\gamma)}\, .
\end{split}
\end{align}
By the monotonicity of $\phi$, the value of the above integral is at least 1 whenever $|\alpha_1 |+ |\alpha_2| \geq 2 \gamma$. Since $\nu(s,t) \geq \nu_1(s,t)$, \eqref{nucond1} is satisfied for such pairs $x_0, y_0$.

Assume now that $|\alpha_1| + |\alpha_2|<2 \gamma$, and $|\alpha_1| \leq |\alpha_2|$. Note that $y_0 = -x_0 \pm \tan \alpha_1$, and therefore
the line containing $(0,y_0)$ and $(1, x_0)$ passes through $(\frac 1 2, \frac{\tan \alpha_1} {2})$ or $(\frac 1 2, -\frac{\tan \alpha_1} {2})$ (see Figure~\ref{fig7b}). Thus, symmetry and convexity of $\psi(t)$ implies that
\[
\int_{0}^1 \nu_2 (s, y_0 + s(x_0 - y_0) ) \dd s \geq  \psi\left(\frac{\tan \alpha_1} {2}\right) = 1 - \frac{\phi(|\alpha_1|)}{\phi(\gamma)}.
\]
Accordingly, by \eqref{nu1int}, and since $|\alpha_1| \leq |\alpha_2|$,
\begin{align*}
&\int_{-1}^0\nu (s, y_0 + s(x_0 + y_0) ) \dd s + \int_{0}^1 \nu (s, y_0 + s(x_0 - y_0) ) \dd s\\
&\geq 2w\phi ( |\alpha_1|)+ 1 - \frac{\phi(|\alpha_1|)}{\phi(\gamma)} = 1.
\end{align*}
Thus, $\nu(s,t)$ satisfies \eqref{nucond1}, and we must show \eqref{nuint1849}. Recall that  $ \textrm{Area}(P_i) = 2 \eps$ for $i = 1, \dots, 4$ and that the horizontal projections of $P_5$ and $P_6$ are of length $\eps$. Therefore,
\begin{align*}
 \int_{[-1,1]^2} \nu(s,t) \dd s \dd t &=  \int_{[-1,1]^2} \nu_1(s,t) \dd s \dd t + \int_{[-1,1]^2} \nu_2(s,t) \dd s \dd t \\
 &= \frac{4}{\phi(\gamma)}  + 2 \int_{-h}^h \psi(t) \dd t\\
 &= \frac{4}{\phi(\gamma)}  + 2 \tan \gamma - \frac 8 {\phi(\gamma)} \arctanh \left( \frac{\tan \gamma}{2}+ \eps \right) \\
 &\approx  \frac{4}{\phi(\gamma)}  + 2 \tan \gamma - \frac 8 {\phi(\gamma)} \arctanh \left( \frac{\tan \gamma}{2}\right).
\end{align*}
On the interval $[0,1]$, the above quantity is minimal at $\gamma= 0.746$ with the attained value of~$1.8485$. Therefore, by setting this value for $\gamma$, for sufficiently small $\eps$, the integral of $\nu(s,t)$ on $[-1,1]^2$ is less than $1.849$.
\end{proof}

Our goal above was to demonstrate that the linear programming method cannot yield a proof for Conjecture~\ref{conj_piercing}, and we did not set off to minimize $\int_{[-1,1]^2} \nu(s,t) \dd s \dd t$ among suitable density functions. By refining the construction, the factor 0.925 of Theorem~\ref{thm_LPupper} can be improved. For example, approximating $[-1,1]^2$ with a $300\times 300$ grid, the discrete density function found by computer search  (see Figure~\ref{fig6b}) yields the upper estimate $0.7915 \, n$.

\section{Higher dimensions}
The analogous questions may be formulated in higher dimensions as well, when we would like to hit or cut (pierce) the cells of the $d$-dimensional box $Q_n^d$ of size $n \times n \times \ldots \times n$ with as few hyperplanes as possible. Higher dimensional analogues of Proposition~\ref{prop_line} were studied in \cite{BF21} and \cite{BF21+}. The authors proved that in the 3-dimensional case, any given plane cuts at most $\frac 9 4 n^2 + 2n +1$ cells, while the upper bound in the $d$-dimensional case is $v_d n^{d-1} (1 + o(1))$ where $v_d \approx \sqrt{6 d /\pi}$ is a well-defined constant. Thus, we derive that the minimum number of hyperplanes needed to cut each cell is at least $\frac 4 9 n (1 + o(1))$ when $d=3$ and at least $\frac 1 {\sqrt{2d}} n (1 + o(1))$ when $d \geq 4$. On the other hand, $n$ parallel hyperplanes clearly suffice.

Turning to the hitting problem, the situation is different: the proof of Theorem ~\ref{thm_hitting} extends to higher dimensions with no difficulty. Therefore, we obtain that the minimal number of hyperplanes needed to hit each cell of $Q_n^d$ is exactly $\lceil \frac n 2 \rceil$.

\medskip
{\bf Acknowledgments.} We are grateful to B. Keszegh and D. Pálvölgyi as well as for the anonymous referees for useful suggestions and for calling our attention to the field of digital geometry.

\bigskip

\noindent
{\sc Gergely Ambrus}
\smallskip

\noindent
{\em Alfréd Rényi Institute of Mathematics, Budapest, Hungary} and\\ {\em Bolyai Institute, University of Szeged, Hungary}
\smallskip

\noindent
e-mail address: \texttt{ambrus@renyi.hu}

\bigskip

\noindent
{\sc Imre Bárány}
\smallskip

\noindent
{\em Alfréd Rényi Institute of Mathematics, Budapest, Hungary}, and \\
{\em Department of Mathematics, University College London, UK}
\smallskip

\noindent
e-mail address: \texttt{barany@renyi.hu}
\bigskip

\noindent
{\sc P\'eter Frankl}
\smallskip

\noindent
{\em Alfréd Rényi Institute of Mathematics, Budapest, Hungary }
\smallskip

\noindent
e-mail address: \texttt{peter.frankl@gmail.com}
\bigskip

\noindent
{\sc Dániel Varga}
\smallskip

\noindent
{\em Alfréd Rényi Institute of Mathematics, Budapest, Hungary }
\smallskip

\noindent
e-mail address: \texttt{daniel@renyi.hu}

\end{document}